\documentclass[pdflatex,sn-mathphys-ay]{sn-jnl}% Math and Physical Sciences Author Year Reference Style
\usepackage{graphicx}%

\usepackage{caption}
\usepackage{subcaption} %

\usepackage{multirow}%
\usepackage{amsmath,amssymb,amsfonts}%
\usepackage{amsthm}%
\usepackage{mathrsfs}%
\usepackage[title]{appendix}%
\usepackage{xcolor}%
\usepackage{textcomp}%
\usepackage{manyfoot}%
\usepackage{booktabs}%
\usepackage{algorithm}%
\usepackage{algorithmicx}%
\usepackage{algpseudocode}%
\usepackage{listings}%
\usepackage{lineno}
\usepackage{setspace}
\newtheorem{thm}{Theorem}[section]
\newtheorem{prop}[thm]{Proposition}

\newtheorem{lem}[thm]{Lemma}
\theoremstyle{definition}

\theoremstyle{remark}

\numberwithin{equation}{section}

\makeatletter % `@' now normal "letter"
\@addtoreset{equation}{section}
\makeatother  % `@' is restored as "non-letter"
\renewcommand\theequation{\oldstylenums{\thesection}%
	.\oldstylenums{\arabic{equation}}}
\newcommand\norm[1]{\lVert1\rVert}
\newcommand\abs[1]{\lvert1\rvert}

\newcommand\set[1]{\left\{1\right\}}

\usepackage{xcolor}

\usepackage{ulem}   % 提供 \sout

\usepackage{ifthen}
\newboolean{showedits}
\setboolean{showedits}{true}   % 这里改成 false 即可隐藏所有标记

\newcommand{\remove}[1]{%
	\ifthenelse{\boolean{showedits}}%
	{\textcolor{red}{\sout{#1}}}%
	{}%
}

\newcommand{\add}[1]{%
	\ifthenelse{\boolean{showedits}}%
	{\textcolor{blue}{#1}}%
	{#1}% final 模式下仍然保留内容
}

\newcommand{\replace}[2]{%
	\ifthenelse{\boolean{showedits}}%
	{\remove{#1}\add{#2}}%
	{#2}% final 模式下只保留新内容
}
\usepackage{footmisc}
\begin{document}

\title[Global stability analysis of a mathematical model from Alzheimer's disease]{Global stability analysis of a mathematical model from Alzheimer’s disease}

%%=============================================================%%
%% GivenName	-> \fnm{Joergen W.}
%% Particle	-> \spfx{van der} -> surname prefix
%% FamilyName	-> \sur{Ploeg}
%% Suffix	-> \sfx{IV}
%% \author*[1,2]{\fnm{Joergen W.} \spfx{van der} \sur{Ploeg} 
%%  \sfx{IV}}\email{iauthor@gmail.com}
%%=============================================================%%

\author[1]{\fnm{Ruoyun} \sur{Lang}}\email{langry@mail.ustc.edu.cn}
\author*[2]{\fnm{Hui} \sur{Zhou}}\email{zhouh16@mail.ustc.edu.cn}
\equalcont{These authors contributed equally to this work.}

\affil[1]{School of Mathematical Sciences, University of Science and Technology of China, Hefei, Anhui 230026, P. R. China}
\affil[2]{School of Mathematics and Statistics, Hefei Normal University, Hefei, Anhui 230601, P. R. China}

%%==================================%%
%% Sample for unstructured abstract %%
%%==================================%%

\abstract{This study focuses on a mathematical model of Alzheimer's disease involving $\beta$-amyloid, cellular prion protein and their complex. The global asymptotic stability of the model indicates that the complex continues to induce neuronal damage regardless of the initial states. To investigate the dynamics of this system, we have rigorously proved that when the formation rate of new plaques is zero, the system is unconditional globally asymptotically stable without any limitation proposed in previous work. Numerical simulations further validate the theoretical analysis, regardless of the random initial state, demonstrating that the system consistently converges to a unique positive equilibrium. From a therapeutic perspective, we propose targeted therapeutic strategies and verify their effectiveness through numerical simulations. These results provide a universal theoretical basis for understanding dynamic mechanisms of Alzheimer's disease and offer critical guidance for developing targeted therapeutics.}


%%\pacs[JEL Classification]{D8, H51}


\footnotetext{H. Zhou is supported by the National Natural Science Foundation of China (No.12517710), and R. Lang is partially supported by the National Natural Science Foundation of China (No.12331006), the Strategic Priority Research Program of CAS (No.XDB0900100) and the National Key R\&D Program of China (No.2024YFA1013603, 2024YFA1013600).}

\maketitle
\section{Introduction}\label{secint}
Alzheimer's disease is one of the most common neurodegenerative diseases among the elderly population \citep{alzfacts2025,masters2015}. Characterized by progressive cognitive decline, brain atrophy, and pathological deposits in the brain \citep{alzheimer1907,stelzmann1995}, it has become a significant health concern facing aging populations worldwide \citep{khan2024,alzfacts2025}. A growing body of evidence indicates that the cerebral accumulation of \textit{$\beta$-amyloid peptide} ($A\beta$) is the primary driver of Alzheimer's disease pathogenesis \citep{hardy2002,herrup2015,jack2013,hardy1992}. Specifically, $A\beta$ is generated via the aberrant cleavage of \textit{amyloid precursor protein} and contributes to Aslzheimer's disease progression through polymorphic aggregation. These \(A\beta\) peptides first aggregate into soluble \(A\beta\) oligomers, which subsequently aggregate into dense, insoluble \(A\beta\) plaques. Among these $A\beta$ forms, oligomers are widely recognized as the principal neurotoxic effectors. Lauren et al. \citep{lauren2009} have demonstrated that \textit{cellular prion protein} (\(\text{PrP}^\text{C}\)) acts as a critical mediator for \(A\beta\) oligomers and induces synaptic dysfunction by disrupting the balance between \(A\beta\) production and clearance. As the global population ages, Alzheimer's disease has emerged as an urgent public health challenge, which highlights the critical need to investigate its underlying pathological mechanisms and develop effective therapeutic interventions \citep{nichols2019}.

%%%%%%%%%%%%

Numerous mathematical models have been developed to characterize \(A\beta\) aggregation dynamics in Alzheimer's disease and to describe the qualitative behavior of this process. Previous studies primarily focus on isolated pathological processes, including the aggregation and clearance kinetics of \(A\beta\) \citep{achdou2013,andrade2021,craft2002,bertsch2016} and the expression and proliferation patterns of \(PrP^{C}\) \citep{pruss2006,hall2004,greer2006,calvez2009}. In order to investigate the interaction mechanism between \(A\beta\) and its neuronal membrane receptor \(PrP^{C}\), Helal et al. \citep{helal2014} proposed an \textit{in vivo} mathematical model introducing the species of the binding as the $A\beta$-x-$PrP^C$ complex. In their formulation, \(f(x, t)\) denotes the density of \(A\beta\) plaques with size $x$ at time $t$ and the size variable $x$ is constrained to \((x_0, +\infty)\). Here, \(x_0\) refers to the minimum size of a stable \(A\beta\) plaque and is quantified as \(x_0 = \varepsilon n\), where \(\varepsilon > 0\) is a positive parameter denoting the mass of a single \(A\beta\) oligomer, and $n$ is a positive integer representing the minimum number of oligomers required to form a stable \(A\beta\) plaque. Let $u(t)$, $p(t)$ and $b(t)$ denote the concentrations of soluble $A\beta$ oligomers, $PrP^C$, and the $A\beta$-x-$PrP^C$ complex, respectively. As consequence, the model takes form as follows:
\begin{align*}\label{1}
	\tag{1.1}
	\begin{cases}
		&\frac{\partial}{\partial t} f(x,t) + u(t) \frac{\partial}{\partial x} [\rho(x) f(x,t)] = -\mu(x) f(x,t) \quad \text{on } (x_0, +\infty) \times (0, +\infty), \\
		&\dot{u} = \lambda_u - \gamma_u u - \tau u p - \sigma b - n N(u) - \frac{1}{\epsilon} u \int_{x_0}^{+\infty} \rho(x) f(x,t) dx \quad \text{on } (0, +\infty), \\
		&\dot{p} = \lambda_p - \gamma_p p - \tau u p + \sigma b \quad \text{on } (0, +\infty), \\
		&\dot{b} = \tau u p - (\sigma + \delta) b \quad \text{on } (0, +\infty),
	\end{cases}
\end{align*}
where

\newcommand{\bcenter}{\mathbin{\vcenter{\hbox{\Large$\bullet$}}}}

$\bcenter \quad  \rho(x):$ the conversion rate of oligomers to plaques;

$\bcenter \quad  \mu(x):$ the plaque degradation rate of oligomers to plaques;

$\bcenter \quad \lambda_u, \lambda_p:$ the sources of $A\beta$ oligomers and $PrP^C$, respectively;

$\bcenter \quad \gamma_u, \gamma_p, \delta:$ the degradation rate of $A\beta$ oligomers, $PrP^C$ and $A\beta$-x-$PrP^{C}$ complex, respectively;

$\bcenter \quad \tau, \sigma:$ the binding and unbinding rate of $A\beta$-x-$PrP^{C}$ complex, respectively.

\noindent  In the second equation of model (\ref{1}), the term \(N(u)\) denotes the formation rate of new \(A\beta\) plaques with critical size \(x_0\), balanced by the boundary condition \(u(t)p(x_0)f(x_0,t)=N\bigl(u(t)\bigr)\), and the term $\int_{x_0}^{+\infty} \rho(x) f(x,t) dx$ represents the total polymerization. The conversion rate \(\rho(x)\) follows the relationship \(\rho(x)\sim x^\theta\) (\(\theta \in (0, 1)\)), owing to the fact that larger \(A\beta\) plaques display stronger binding affinity for \(A\beta\) oligomers, with the increase in this affinity being suppressed by the saturation of surface binding sites. All parameters \(\lambda_u, \lambda_p, \gamma_u$, $\gamma_p, \delta, \sigma, \text{and } \tau\) are positive, all functions \(\rho(x), \mu(x), \text{and } N(u)\) are non-negative, and all initial values in the model are non-negative (i.e., \(f(x,0) \ge 0\) for almost all \(x > x_0\), \(u(0) \ge 0\), \(p(0) \ge 0\) and \(b(0) \ge 0\)). Under the assumptions \(\mu(x) \in W^{1,\infty}([x_0, \infty))\) and \(N(u) \in W^{1,\infty}_{\text{loc}}(\mathbb{R}_+)\) with \(N(0)=0\), Helal et al. \citep{helal2014} prove the well-posedness of model (1), where \(W^{1,\infty}([x_0, \infty))\) is the global first-order essentially bounded Sobolev space on \([x_0, \infty)\) and \(W^{1,\infty}_{\text{loc}}(\mathbb{R}_+)\) is the local counterpart on \(\mathbb{R}_+\) ($\mathbb{R}_+$ denotes the set of all positive real numbers).

In the brain, the total population of \(A\beta\) plaques which defined as \(A=\int_{x_0}^{+\infty} f(x,t) dx\) is one of the core observable indicators \citep{destrooper2016}. When focusing on the total \(A\beta\) plaques, the conversion rate of oligomers to plaques \(\rho(x)\) and plaques degradation rate \(\mu(x)\) can be simplified to constant values, which reflect the average dynamic characteristics of total \(A\beta\) plaques population. Considering multi-molecular cooperative aggregation, De Strooper and Karran propose the new plaque formation term \(N(u) \) can be expresses by \(N(u) = \alpha u^n\), where non-negative constant \(\alpha \) denotes the average formation rate to a new plaque. Without loss of generality, one can normalize the mass parameter \(\varepsilon\) of a single \(A\beta\) oligomer to 1. By integrating the first equation of system (\ref{1}) over \((x_0, +\infty)\) and focusing on the key variables \(A\beta\) plaques, $A\beta$ oligomers, $PrP^C$ and $A\beta$-x-$PrP^C$ complex, model (1) can be turned out to be the following system:
\begin{align*}\label{Aupb}			
	\tag{1.2}			
	\begin{cases}				
		\dot{A} &= \alpha u^n - \mu A, \\				
		\dot{u} &= \lambda_u - \gamma_u u - \tau u p + \sigma b - \alpha n u^n - \rho u A, \\				
		\dot{p} &= \lambda_p - \gamma_p p - \tau u p + \sigma b, \\				
		\dot{b} &= \tau u p - (\sigma + \delta) b.			
	\end{cases}		
\end{align*}

For system (\ref{Aupb}), Helal et al. \citep{helal2014} obtain that the positive equilibrium $\tilde{E} = (\tilde{A}, \tilde{u}, \tilde{p}, \tilde{b})$ (see section \ref{secmain} for its explicit expression) always exist and is unique, and then $\tilde{E}$ is locally asymptotically stable. Based on the local stability of the unique equilibrium and its correspondence to the inherent dynamic balance between \(A\beta\) production and clearance in the brain, we conjecture that \(\tilde{E}\) is globally asymptotically stable. Moreover, it is crucial to verify this conjectured global asymptotic stability for understanding the long-term regulatory mechanisms governing \(A\beta\) aggregation in Alzheimer’s disease pathology. However, for 4-dimensional system (\ref{Aupb}), it is a challenging task to prove the specific system is globally asymptotically stable. For this purpose, considering that the new plaque formation is absent in the early stage of Alzheimer’s disease (i.e., \(\alpha = 0\)), Helal et al. establish a \textit{conditional} global asymptotic stability for system (\ref{Aupb}) by imposing the following parameter constraint:
\begin{flalign*}
	\label{H}
	\text{\textbf{(H)}} \quad \left( 1 + 2 \dfrac{\delta + \gamma_u}{\sigma} \right) > \frac{\delta}{2\gamma_p} > \dfrac{\gamma_p}{\sigma}. &&
\end{flalign*}

%%%%%%%%%%%%%%%%%%
\noindent Notably, this proof is achieved by means of the Lyapunov function method, and it is crucial to clarify that condition (\textbf{H}) originates not from the underlying intrinsic rules of \(A\beta\) aggregation dynamics in Alzheimer’s disease, but from the technical limitations inherent of this method. In a biological context, parameters exhibit substantial variability across different experimental settings and physiological states, and this variability renders strict compliance with this constraint impractical in practice. Our numerical simulations, which are detailed in section \ref{sec4.1}, demonstrate that when this condition is completely removed, the system (\ref{Aupb}) still converges to a unique stable equilibrium. These findings directly confirm that the system exhibits $unconditional$ global asymptotic stability and consequently that the parameter constraints included in condition (\textbf{H}) are unreasonable.

Accordingly, this paper focuses on investigating the $unconditional$ global asymptotic stability of the equilibrium ror system (\ref{Aupb}) with \(\alpha = 0\). Based on the structural properties of system (\ref{Aupb}), the system reduces partially to a three-dimensional K-competitive system in the absence of the new plaque formation rate. Given that three-dimensional K-competitive systems possess the Poincaré–Bendixson property, the challenging task of proving the global asymptotic stability of system (\ref{Aupb}) is simplified to verifying the orbital stability of any potential positive periodic solutions within the relevant positive invariant set. To this end, we utilize the stability criterion based on \textit{second compound equations} proposed by Muldowney \citep{muldowney1990}, thereby establishing our main results.

\renewcommand{\thethm}{A}
\begin{thm}\label{thm1}
	For system (\ref{Aupb}), $\tilde{E}$ is always globally asymptotically stable.
\end{thm}
\renewcommand{\thethm}{\thesection.\arabic{thm}}

The \textit{unconditional} global asymptotic stability of the unique positive equilibrium, established in Theorem A, holds critical biological and clinical implications for Alzheimer’s disease. Specifically, this equilibrium corresponds to a pathological steady state marked by persistent \(A\beta\) plaques, soluble \(A\beta\) oligomers, \(PrP^C\), and neurotoxic $A\beta$-x-$PrP^C$ complex, which implies the central nervous system is impaired. The global stability of the equilibrium indicates that while interventions may temporarily alter these molecular concentrations, the system will inevitably converge to this pathological steady state. This finding provides a critical therapeutic implication, specifically that instead of striving for a complete cure for Alzheimer’s disease, therapeutic strategies should modulate core kinetic parameters to shift this equilibrium toward a state characterized by minimized neurotoxic species and attenuated neuronal damage.

To validate the theoretical findings and explore therapeutic implications, we further perform specialized numerical simulations in section \ref{sec4.2} targeting the \(A\beta\)-\(PrP^C\) aggregation dynamics of system (\ref{Aupb}). The results clearly reveal the differential effects of various targeting strategies. Although anti-amyloid antibodies commonly used in clinical studies can effectively reduce the formation of soluble oligomers by specifically binding to free amyloid molecules, their binding affinity to preformed $A\beta$-x-$PrP^C$ complex is significantly reduced, resulting in negligible inhibitory effects on toxic complex accumulation. Shifting the target to \(PrP^C\) itself or the $A\beta$-x-$PrP^C$ complex enables more efficient clearance of existing toxic species and suppression of new complex generation at the source. Crucially, by lifting the restrictive parameter condition (\textbf{H}), our analysis of system (\ref{Aupb}) overcomes constraints on parameter ranges and provides critical theoretical support for the development of targeted, long-acting Alzheimer’s disease therapeutic strategies.

The remainder of the paper is organized as follows. In section \ref{secmain}, the main results of this paper are presented. In section \ref{secprf}, we provide a rigorous proof of the main conclusion using monotone dynamical system theory and compound matrix methods. In section \ref{secsim}, we analyze the effects of key parameters on the dynamics through numerical simulations. Section \ref{secdis} is the discussion.

%%%%%%%%%%%%%%%%%%%%%%%%%%%%%%%%%%%%%%%%%%%%%%%%%%%%%%%%%%%%%%
%%%%%%%%%%%%%%%%%%%%%%%%%%%%%%%%%%%%%%%%%%%%%%%%%%%%%%%%%%%%%%
%%%%%%%%%%%%%%%%%			Main Results
%%%%%%%%%%%%%%%%%%%%%%%%%%%%%%%%%%%%%%%%%%%%%%%%%%%%%%%%%%%%%%
%%%%%%%%%%%%%%%%%%%%%%%%%%%%%%%%%%%%%%%%%%%%%%%%%%%%%%%%%%%%%%
\section{Main Results}\label{secmain}
We mainly focus on the global dynamics of system (\ref{Aupb}). Let us define in the present paper our working domain $\Sigma$ as
\begin{equation}\label{Sigma}
	\Sigma= \left\{ (A, u, p, b) \in \mathbb{R}_{+}^{4} : nA + u + p + 2b \leq  \frac{\lambda}{m} \right\}
\end{equation}
where $\lambda = \lambda_{u}+\lambda_{p}$ and $m = \min\{\mu, \gamma_{u}, \gamma_{p}, \delta\}$. For the validity of the choice of $\Sigma$ from a mathematical point of view, one can see more details in \cite[Proposition 1]{helal2014}. Moreover, there exists a unique positive equilibrium $\tilde{E} = (\tilde{A}, \tilde{u}, \tilde{p}, \tilde{b})$ in $\Sigma$ with 

\[ \tilde{A} = \frac{\alpha\tilde{u}^n}{\mu} , \quad \tilde{p} = \frac{\lambda_p }{( \tau^* \tilde{u} + \gamma_p )},\quad   \tilde{b} = \frac{\lambda_p (\tau - \tau^*)\tilde{u}}{\sigma  (\tau^* \tilde{u}+ \gamma_p)}  ,  \]
where $\tau^* = \tau (1 - \frac{\sigma}{\delta + \sigma} )$ and $\tilde{u}$ is the unique positive root of the function $ M(x) = \gamma_p \lambda_u + ax - N(x)$ with $a = \tau^* (\lambda_u - \lambda_p) - \gamma_u \gamma_p$ and $N(x) = \tau^* \gamma_u x^2 + \alpha \gamma_p n x^n + (\alpha \tau^* n + p \gamma_p \frac{\alpha}{\mu}) x^{n+1} + \rho \tau^* \frac{\alpha}{\mu} x^{n+2}$.

In the early stages of Alzheimer's disease, the pathological process is characterized by the initial accumulation of key toxic molecules rather than the massive formation of mature pathological structures. Within the framework of system (\ref{Aupb}), the parameter $\alpha$ is defined as the rate at which soluble $A\beta$ oligomers aggregate to form new insoluble plaques. In consideration of the negligible biological significance of plaque formation in the early stages of Alzheimer's disease, the following system is derived from the model upon setting $\alpha = 0$:

\begin{equation}\label{A0upb}
	\begin{cases}
		\dot{u} = \lambda_u  - \gamma_u u - \tau u p + \sigma b - \rho u A_0 e^{-\mu t}, \\
		\dot{p} = \lambda_p - \gamma_p p - \tau u p + \sigma b, \\
		\dot{b} = \tau u p - (\sigma + \delta) b
	\end{cases}
\end{equation}
with $A(t) = A_0 e^{-\mu t}$. Sence the plaque term vanishes
\[ \lim_{t \to \infty} A(t) = \lim_{t \to \infty}A_0 e^{-\mu t} = 0, \]
i.e. \( \lim\limits_{t \to \infty} -\rho u A_0 e^{-\mu t} = 0 \), the limiting system of (\ref{A0upb}) is

\begin{equation}\label{upb}
	\begin{cases}
		\dot{u} = \lambda_u  - \gamma_u u - \tau u p + \sigma b , \\
		\dot{p} = \lambda_p - \gamma_p p - \tau u p + \sigma b, \\
		\dot{b} = \tau u p - (\sigma + \delta) b.
	\end{cases}
\end{equation}
Consequently, the domain $\Sigma$ reduces to the compact invariant set
\begin{equation}\label{Gamma}
	\Gamma= \left\{ (u, p, b) \in \mathbb{R}_+^3 : u + p + 2b \leq \frac{\lambda}{\tilde{m}} \right\},
\end{equation}
where $\tilde{m}= \min\{ \gamma_{u}, \gamma_{p}, \delta\}\geq m$, which is positively invariant with respect to system (\ref{upb}). The following proposition confirms the validity
of the choice of $\Gamma$ from a mathematical point of view.

\begin{prop}\label{prop2.1}
	The set $\Gamma$ is positively invariant with respect to system (\ref{upb}). Moreover, for any initial value \((u_0, p_0, b_0) \in \mathbb{R}^3_+\), the solutions are strictly positive for all $t > 0$ and remain in the positively invariant set $\Gamma$.
\end{prop}
Regarding the stability of system (\ref{upb}), we can obtain the following Theorem \ref{thm2.2}.
\begin{thm}\label{thm2.2}		 		 
	There exists a unique positive steady state \(E^* = (u^*, p^*, b^*)\) of system (\ref{upb}), moreover, this equilibrium is globally asymptotically stable. 		 
\end{thm}
Building upon the above global asymptotic stability results for the system (\ref{upb}), we can extend the analysis to the system (\ref{Aupb}). The stability conclusion regarding the equilibrium of system (\ref{Aupb}) is stated in Theorem \ref{thm2.3}.
\begin{thm}\label{thm2.3}
	For system (\ref{Aupb}), the equilibrium $\tilde{E} = (\tilde{A}, \tilde{u}, \tilde{p}, \tilde{b})$ is always globally asymptotically stable with respect to $\Sigma$.
\end{thm}	

%%%%%%%%%%%%%%%%%%%%%		

%%%%%%%%%%%%%%%%%%%%%%%%%%%%%%%%%%%%%%%%%%%%%%%%%%%%%%%%%%%%%%
%%%%%%%%%%%%%%%%%%%%%%%%%%%%%%%%%%%%%%%%%%%%%%%%%%%%%%%%%%%%%%
%%%%%%%%%%%%%%%%%			Proof
%%%%%%%%%%%%%%%%%%%%%%%%%%%%%%%%%%%%%%%%%%%%%%%%%%%%%%%%%%%%%%
%%%%%%%%%%%%%%%%%%%%%%%%%%%%%%%%%%%%%%%%%%%%%%%%%%%%%%%%%%%%%%
\section{Proof of Main Results}\label{secprf}

In this section, we give the proof of main results. To establish the positive invariance of \(\Gamma\) and the positivity of solutions to system \eqref{upb}, we provide the following proof. 

\begin{proof}[Proof of Proposition 2.1]
	Let (u,p,b) be any solution of system (\ref{upb}) with the initial value $(u(0),p(0),b(0))\in \Gamma$. By the form of system (\ref{upb}), it is straightforward to verify that $(u(t),p(t),b(t)) \in \mathbb{R}^3_+$ for any $t\geq 0$. Now, define the function  \(Q(t)=u(t)+p(t)+2b(t)\). According to system (\ref{upb}), we note that
	\begin{align*}
		\dot{Q}&=\dot{u}+\dot{p}+2\dot{b}\\
		&=(\lambda_u-\gamma_uu-\tau up+\sigma b)+(\lambda_p-\gamma_pp-\tau up+\sigma b)+2(\tau up - (\sigma+\delta)b)\\
		&=\lambda_u+\lambda_p-\gamma_uu-\gamma_pp - 2\delta b\\
		&\leq\lambda-\tilde{m}(u + p + 2b)\\
		&=\lambda-\tilde{m}Q.
	\end{align*}
	This leads to inequality \(\dot{Q} \leq \lambda - \tilde{m}Q\), which can be deformed into 
	\begin{equation*}
		\dot{Q} + \tilde{m}Q \leq \lambda.
	\end{equation*}
	Multiplying both sides simultaneously by \(e^{\tilde{m}t}\), then
	\begin{equation*}
		\frac{d}{dt}\Big(Q(t)e^{\tilde{m}t}\Big) = \dot{Q}e^{\tilde{m}t} + \tilde{m}Qe^{\tilde{m}t} \leq \lambda e^{\tilde{m}t}.
	\end{equation*}	
	Integrating from 0 to \(t\) for both sides of the above inequality,	
	
	\[ 	\int_0^t \frac{d}{ds}\Big(Q(s)e^{\tilde{m}s}\Big)  ds \leq \int_0^t \lambda e^{\tilde{m}s} ds,  \]
	then $Q(t)e^{\tilde{m}t} - Q(0) \leq \frac{\lambda}{\tilde{m}}(e^{\tilde{m}t} - 1)$, which implies to
	\begin{align*}
		Q(t) & \leq Q(0) e^{-\tilde{m}t} + \frac{\lambda}{\tilde{m}}(1 - e^{-\tilde{m}t})\\
		& = (Q(0)-\frac{\lambda}{\tilde{m}})e^{-\tilde{m}t} +\frac{\lambda}{\tilde{m}},
	\end{align*}
	we can get
	\[Q(t) \to  \frac{\lambda}{\tilde{m}}, \quad t\to \infty. \]		
	Thus, there exists \(T > 0\) such that when \(t \geq T\), $Q(t)\leq \frac{\lambda}{\tilde{m}}$ for any \((u(t), p(t), b(t)) \in \mathbb{R}^3_+\), therefore \(\Gamma\) is a positive invariant set with respect to system (\ref{upb}). We have completed the proof.
\end{proof}

Next, let us continue to prove Theorem \ref{thm2.2}. For the system (\ref{upb}), it is of so-called $K$-competitive system with respect to the partial ordering defined by the convex cone \( K = \{(u, p, b) \in \mathbb{R}^3 :u \geq 0, p \geq 0, b \leq 0\}\) in $\mathrm{R}^3$. Since three-dimensional $K$-competitive systems possess the Poincaré-Bendixon property \citep{smith1995}, the proof of global asymptotic stability property of the equilibrium translates into the orbital stability of any possible orthonormal periodic solution in \(\Gamma\). To this end, we implement the stability criterion developed by Muldowney \citep{muldowney1990} based on the second compound matrices.

To get the equilibrium of the system (\ref{upb}), one need to consider the following equation
\begin{equation*}
	\begin{cases}
		\lambda_u  - \gamma_u u - \tau u p + \sigma b = 0 , \\
		\lambda_p - \gamma_p p - \tau u p + \sigma b = 0, \\
		\tau u p - (\sigma + \delta) b = 0.
	\end{cases}
\end{equation*}
It can be proven that there exists a unique positive equilibrium $E^* = (u^*, p^*, b^*)$ defined by the following relationships
\[ 
b^* = \dfrac{\tau u^*p^*}{\sigma + \delta}, \quad u^* = \dfrac{\lambda_u}{\gamma_u + \tau p^*k}
\] 
and $p^*$ is the unique positive root of $P(x)$ defined by
\[ 
P(x)=\gamma_p \tau k x^2 + (\gamma_p \gamma_u + \tau \lambda_u k - \lambda_p \tau k)x - \lambda_p \gamma_u, \quad   x \geq 0. 
\]
Moreover, the explicit expression for $p^{*}$ is 
\[ p^{*}=\dfrac{1}{2\gamma_{p}\tau k}\left(-(\gamma_{p}\gamma_{u}+\tau\lambda_{u}k - \lambda_{p}\tau k)+\sqrt{\Delta} \right), \]
where $k=1 - \frac{\sigma}{\delta + \sigma} $ and $\Delta=(\gamma_{p}\gamma_{u}+\tau\lambda_{u}k - \lambda_{p}\tau k)^{2}+4\gamma_{p}\tau k\lambda_{p}\gamma_{u}$.

Now, we can prove that $E^*$ is locally asymptotically stable from the following Lemma \ref{4.3}.

\begin{lem}\label{4.3}
	The unique equilibrium $E^* = (u^*, p^*, b^*)$ of system (\ref{upb}) is locally asymptotically stable.
\end{lem}
\begin{proof}
	To analyze the stability, we calculate the Jacobian matrix \(J\) at \(E^*\), as follows
	\begin{equation}\label{J}
		J(E^*)=\begin{bmatrix}
			-\gamma_u-\tau p^* & -\tau u^* & \sigma\\
			-\tau p^* & -\gamma_p-\tau u^* & \sigma\\
			\tau p^* & \tau u^* & -(\sigma+\delta)
		\end{bmatrix}.
	\end{equation}
	
	The characteristic equation of $J(E^*)$ is
	\begin{equation}\label{lambda}
		\lambda^3 + a_1\lambda^2 + a_2\lambda + a_3 = 0,	
	\end{equation}			
	where 
	\begin{align*}
		a_1 &= \gamma_u + \gamma_p + \sigma + \delta + \tau(u^* + p^*), \\
		a_2 &= \gamma_u(\gamma_p + \sigma + \delta) + \gamma_p(\sigma + \delta) + \tau u^*(\gamma_u + \delta) + \tau p^*(\gamma_p + \delta), \\
		a_3 &= \delta \tau (\gamma_u u^* + \gamma_p p^*) + \gamma_u \gamma_p (\sigma + \delta).
	\end{align*}
	A direct calculation showed that
	\begin{align*}
		a_1a_2 - a_3= &\left(  \gamma_p + \sigma +  \tau(u^* + p^*)\right) \left( \gamma_u(\gamma_p + \sigma + \delta)  + \tau u^*\delta + \tau p^*\delta\right)\\
		&+ \gamma_u(\tau u^*\gamma_u + \tau p^*\gamma_p )+\delta\gamma_p(\sigma + \delta).
	\end{align*}	
	Based on the equilibrium relations and the positivity of the parameters, we have \(a_1>0\), \(a_3>0\) and \(a_1a_2 - a_3>0\). According to Routh-Hurwitz criterion, it follows that all the roots of (\ref{lambda}) have negative real parts. Hence, the equilibrium \(E^*\) is locally linearly asymptotically stable. The proof is completed.
\end{proof}

From (\ref{J}), the system (\ref{upb}) is a $K$-competitive system with respect to the partial ordering defined by the special convex cone \( K \). By the Poincaré-Bendixson theorem for three-dimensional competitive systems \citep{smith1995}, the \(\omega\)-limit set of any trajectory in \(\Gamma\) must be an equilibrium point or a non-constant periodic orbit. Consequently, the demonstration of the absence of periodic trajectories within the system is pivotal in substantiating the global asymptotic stability of the equilibrium. Until such a time as the aforementioned condition is demonstrated to be true, the compound matrix method is utilized to demonstrate that the periodic trajectory must be asymptotically stable, if it exists.

\begin{lem}\label{4.5}
	Any nonconstant periodic solution of system (\ref{upb}) is, if it exists, asymptotically orbitally stable.
\end{lem}
\begin{proof}
	Let $(u(t), p(t), b(t))$ be a periodic solution of system (\ref{upb}) with minimal period $\omega > 0$. According to second compound matrices method developed by Muldowney \citep{muldowney1990}, the orbital asymptotic stability of this periodic solution is equivalent to the asymptotic stability of the associated second additive compound linear system		 
	\begin{equation}\label{3.2}
		\frac{dX}{dt} = J^{[2]}_{upb}(u, p, b)X,
	\end{equation}		 
	where $X = (X_1, X_2, X_3)^T$ and $J^{[2]}_{upb}$ is the second additive compound matrix of the Jacobian of the system (\ref{upb}) as follows
	
	\begin{equation*}
		J^{[2]}_{upb} = \begin{pmatrix}
			-(\gamma_u + \gamma_p + \tau(p + u)) & \sigma & -\sigma \\
			\tau u & -(\gamma_u + \sigma + \delta + \tau p) & -\tau u \\
			-\tau p & -\tau p & -(\gamma_p + \sigma + \delta + \tau u)
		\end{pmatrix}.
	\end{equation*}
	Define the function  $W(t) = |X_1| + |X_2| + |X_3|$, then $D_+W(t) = \sum_{i=1}^{3} D_+(|X_i|)$. With the inequality \(\pm \text{sign}(X_i) X_j \leq |X_j|\) (for any \(i \neq j\)), the specific expression for each \(D_+(|X_i|)\) is as follows,	
	\begin{align*}
		D_+(|X_1|)&= \text{sign}(X_1) \cdot \frac{dX_1}{dt} = \text{sign}(X_1) \left[ J_{11}^{[2]}X_1 + J_{12}^{[2]}X_2 + J_{13}^{[2]}X_3\right]\\	
		&=\text{sign}(X_1) \left[ -(\gamma_u + \gamma_p + \tau(p+u))X_1 + \sigma X_2 - \sigma X_3 \right]  \\
		&= -(\gamma_u + \gamma_p + \tau(p+u))|X_1| + \sigma \cdot \text{sign}(X_1)X_2 - \sigma \cdot \text{sign}(X_1)X_3\\
		&\leq -(\gamma_u + \gamma_p + \tau(p+u))|X_1| + \sigma |X_2| + \sigma |X_3|,
	\end{align*}		
	\begin{align*}
		D_+(|X_2|) &= \text{sign}(X_2) \cdot \frac{dX_2}{dt} = \text{sign}(X_2) \left[ J_{21}^{[2]}X_1 + J_{22}^{[2]}X_2 + J_{23}^{[2]}X_3\right]\\
		&= \text{sign}(X_2) \left[ \tau u X_1 - (\gamma_u + \sigma + \delta + \tau p) X_2 - \tau u X_3 \right] \\
		&= \tau u \cdot \text{sign}(X_2) X_1 - (\gamma_u + \sigma + \delta + \tau p) |X_2| - \tau u \cdot \text{sign}(X_2) X_3\\
		&\leq \tau u |X_1| - (\gamma_u + \sigma + \delta + \tau p) |X_2| + \tau u |X_3|
	\end{align*}
	and
	\begin{align*}
		D_+(|X_3|) &= \text{sign}(X_3) \cdot \frac{dX_3}{dt} = \text{sign}(X_3) \left[ J_{31}^{[2]}X_1 + J_{32}^{[2]}X_2 + J_{33}^{[2]}X_3\right]\\
		&= \text{sign}(X_3) \left[ -\tau p X_1 - \tau p X_2 - (\gamma_p + \sigma + \delta + \tau u) X_3 \right] \\
		&= -\tau p \cdot \text{sign}(X_3) X_1 - \tau p \cdot \text{sign}(X_3) X_2 - (\gamma_p + \sigma + \delta + \tau u) |X_3|\\
		&\leq \tau p |X_1| + \tau p |X_2| - (\gamma_p + \sigma + \delta + \tau u) |X_3|.
	\end{align*}
	Summing these inequalities yields
	\begin{align*}
		D_+W(t) \leq&  -(\gamma_u + \gamma_p + \tau(p+u))|X_1| + \sigma |X_2| + \sigma |X_3|  \\
		&+  \tau u |X_1| - (\gamma_u + \sigma + \delta + \tau p)|X_2| + \tau u |X_3|  \\
		&+\tau p |X_1| + \tau p |X_2| - (\gamma_p + \sigma + \delta + \tau u)|X_3|\\
		=&-(\gamma_u + \gamma_p)|X_1|-(\gamma_u + \delta )|X_2|-(\gamma_p + \delta )|X_3|\\
		\leq& -CW(t) 
	\end{align*}
	where \(C = \min\{\gamma_u + \gamma_p, \gamma_u + \delta, \gamma_p + \delta\} > 0\). The differential inequality $D_{+}W(t) \le -CW(t)$ implies $W(t) \le W(0)e^{-Ct}$. We can obtain that $W(t) \rightarrow 0$ as $t \rightarrow \infty$. Hence, the linear system (\ref{3.2}) is asymptotically stable and the periodic solution $(u(t), p(t), b(t))$ is asymptotically orbitally stable by \citep[Theorem 4.2]{muldowney1990}. We have completed the proof.
\end{proof}	
We give the proof of the Theorem \ref{thm2.2} based on the above Lemmas.
\begin{proof}[Proof of Theorem 2.2]
	Following the form of system (\ref{upb}), $\Gamma \subseteq \text{Int}(\mathbb{R}^3_+)$ is a bounded positively invariant region for the three-dimensional system (\ref{upb}). Let $\mathcal{A}$ be the basin of attraction of the locally asymptotically stable equilibrium $E^*$. Then $\mathcal{A}$ is a nonempty relatively open subset of $\Gamma$. Denote $\partial_\Gamma \mathcal{A}$ the boundary of $\mathcal{A}$ relative to $\Gamma$. Clearly $\partial_\Gamma \mathcal{A}$ is invariant. Let $x=(u,p,b) \in \partial_\Gamma \mathcal{A}$. Then the omega-limit set $\omega(x) \subseteq \partial_\Gamma \mathcal{A}$ and $E^* \notin \omega(x)$. Since $E^*$ is the unique positive equilibrium in $\Gamma$. By Poincaré-Bendixson Theorem, $\omega(x) \subseteq \partial_\Gamma \mathcal{A}$, $\omega(x)$ is a periodic orbit of system (\ref{upb}).
	
	From the assumption, $\omega(x)$ is orbitally asymptotically stable. One can choose a point $a \in \mathcal{A}$ sufficiently close to $\omega(x)$. On one hand, $a$ is attracted to $E^*$. On the other hand, $a$ will be asymptotic to the orbitally stable periodic orbit $\omega(x)$, a contradiction. Thus we prove that $E^*$ is globally asymptotically stable in $\Gamma$. The proof is completed.
\end{proof}

Now, we give the proof of Theorem \ref{thm2.3}.
\begin{proof}[Proof of Theorem 2.3]
For system (\ref{Aupb}), $\tilde{A} = \alpha\tilde{u}^n/\mu=0$ when $\alpha = 0$. Moreover system (\ref{upb}) is the part of the limiting system of system (\ref{Aupb}). By Theorem \ref{thm2.2}, the system (\ref{upb}) is globally asymptotically stable. It is widely known that, if system is globally asymptotically stable, all compact, connected, chain recurrent, and invariant sets D must necessarily be singleton sets consisting of this equilibrium. For system (\ref{upb}), $D=\{E^*\}$, by \citep[Theorem 3.1]{mischaikow1995}, D is the \(\omega\)-limit set of system (\ref{A0upb}). With $\lim_{t \to \infty} A(t) = 0$, the \(\omega\)-limit set of system (\ref{Aupb}) is also a singleton set, i.e., the unique positive equilibrium $\tilde{E} = (\tilde{A}, \tilde{u}, \tilde{p}, \tilde{b})$ in $\Sigma$ is globally asymptotically stable. The proof is now complete.
\end{proof}
%%%%%%%%%%%%%%%%%%%%%%%%%%%%%%%%%%%%%%%%%%%%%%%%%%%%%%%%%%%%%%
%%%%%%%%%%%%%%%%%%%%%%%%%%%%%%%%%%%%%%%%%%%%%%%%%%%%%%%%%%%%%%
%%%%%%%%%%%%%%%%%			Discussions
%%%%%%%%%%%%%%%%%%%%%%%%%%%%%%%%%%%%%%%%%%%%%%%%%%%%%%%%%%%%%%
%%%%%%%%%%%%%%%%%%%%%%%%%%%%%%%%%%%%%%%%%%%%%%%%%%%%%%%%%%%%%%
\section{Numerical Simulations}\label{secsim}
%%%%%%%%%%%%%%%%%%%%%%%%%%%%%%%%%%%%%%%%%%%%%%%%%%%%%%%%%%%%%%
%%%%%%%%%%%%%%%%%%%%%%%%%%%%%%%%%%%%%%%%%%%%%%%%%%%%%%%%%%%%%%	

In this section, to further validate the theoretical findings and explore the dynamic behaviors of system (\ref{Aupb}) under biological scenarios, we conduct numerical simulations using the adaptive step size 4/5th Order Runge-Kutta-Fehlberg method, which ensures high precision in solving the system of ordinary differential equations. The central focus of these simulations is two aspects. Firstly, we seek to verify the global asymptotic stability of the equilibrium under the condition that no new plaques formed. Secondly, we aim at evaluating the therapeutic effects of various targeted drugs, utilizing a methodical adjustment of the relevant parameters. Unless specified otherwise, all simulations are performed with the following parameters: \(\mu=0.01\), \(\lambda_{u}=0.8\), \(\gamma_{u}=0.5\), \(\tau=0.36\), \(\sigma=0.2\), \(\rho=0.006\), \(n=3\), \(\lambda_{p}=0.6\), \(\delta=0.12\), and \(\gamma_{p}=0.5\). Notably, this parameter set violates condition (\textbf{H}), as calculation gives $ 1 + 2 (\delta + \gamma_u) / \sigma =7.20 $, $ \delta / 2\gamma_p=0.12 $ and $ \gamma_p / \sigma=2.50$, which does not satisfy the inequality chain in (\textbf{H}).

\subsection{Validation of Global Asymptotic Stability}\label{sec4.1}	
To confirm the global asymptotic stability of the equilibrium in the absence of new \(A\beta\) plaque formation, i.e., \(\alpha=0\), numerical simulations demonstrate that the dynamics of \(PrP^C\), \(A\beta\) oligomers , and \(A\beta\)-x-\(PrP^C\) complex using 1000 sets of random initial conditions, where the initial concentrations of $p$, $u$, and $b$ were uniformly sampled within the range (0, 2). This range of initial conditions is chosen to reflect the variability of pathological states observed in Alzheimer's disease.
\begin{figure}[htbp]
	\centering
	% 第一张子图（左）
	\begin{subfigure}[b]{0.45\linewidth}  % [b]底部对齐，宽度0.45倍文本宽度
		\centering
		\includegraphics[width=\linewidth]{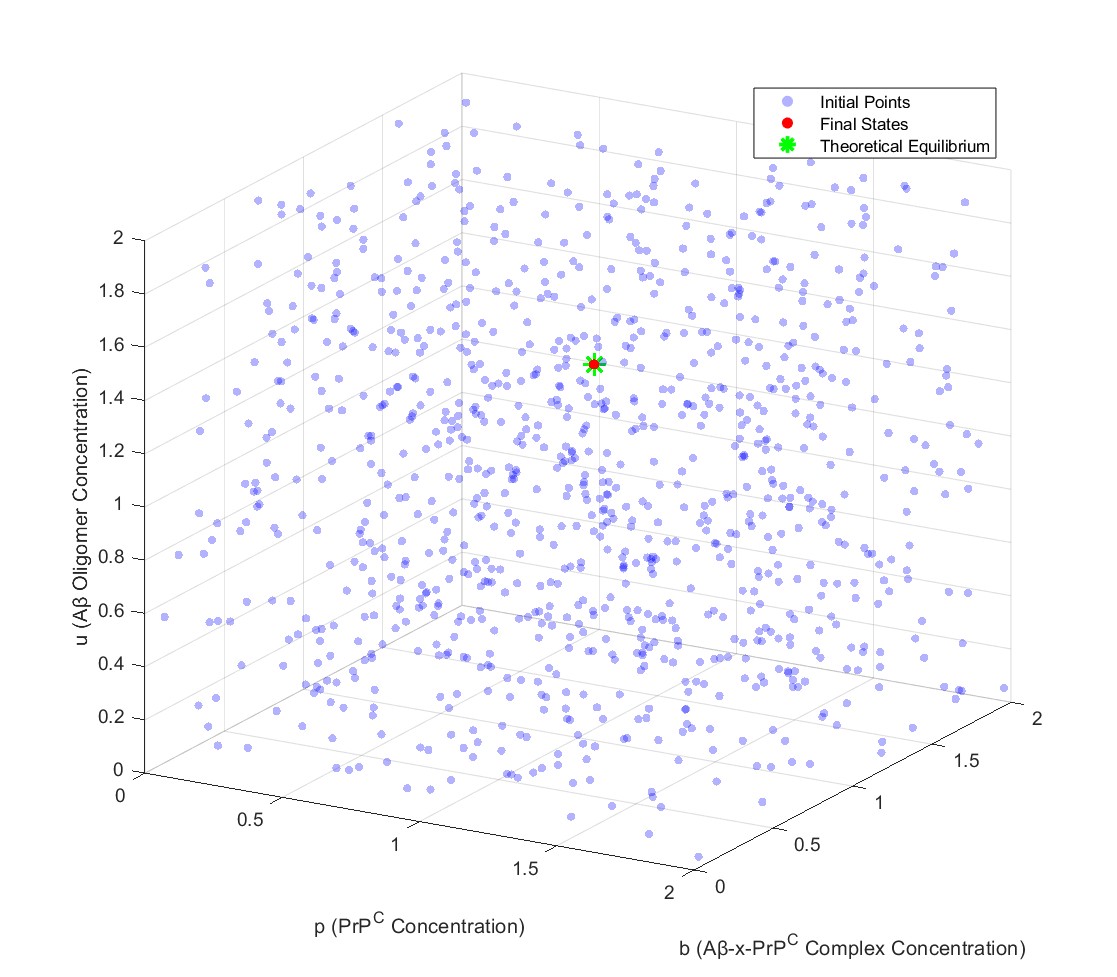}  % 子图宽度占满subfigure
		\caption{}  % 子图标题（自动编号为(a)(b)...）
		\label{sub1}
	\end{subfigure}
	\hfill  % 子图之间的水平间距（可改为\quad等）
	% 第二张子图（右）
	\begin{subfigure}[b]{0.45\linewidth}
		\centering
		\includegraphics[width=\linewidth]{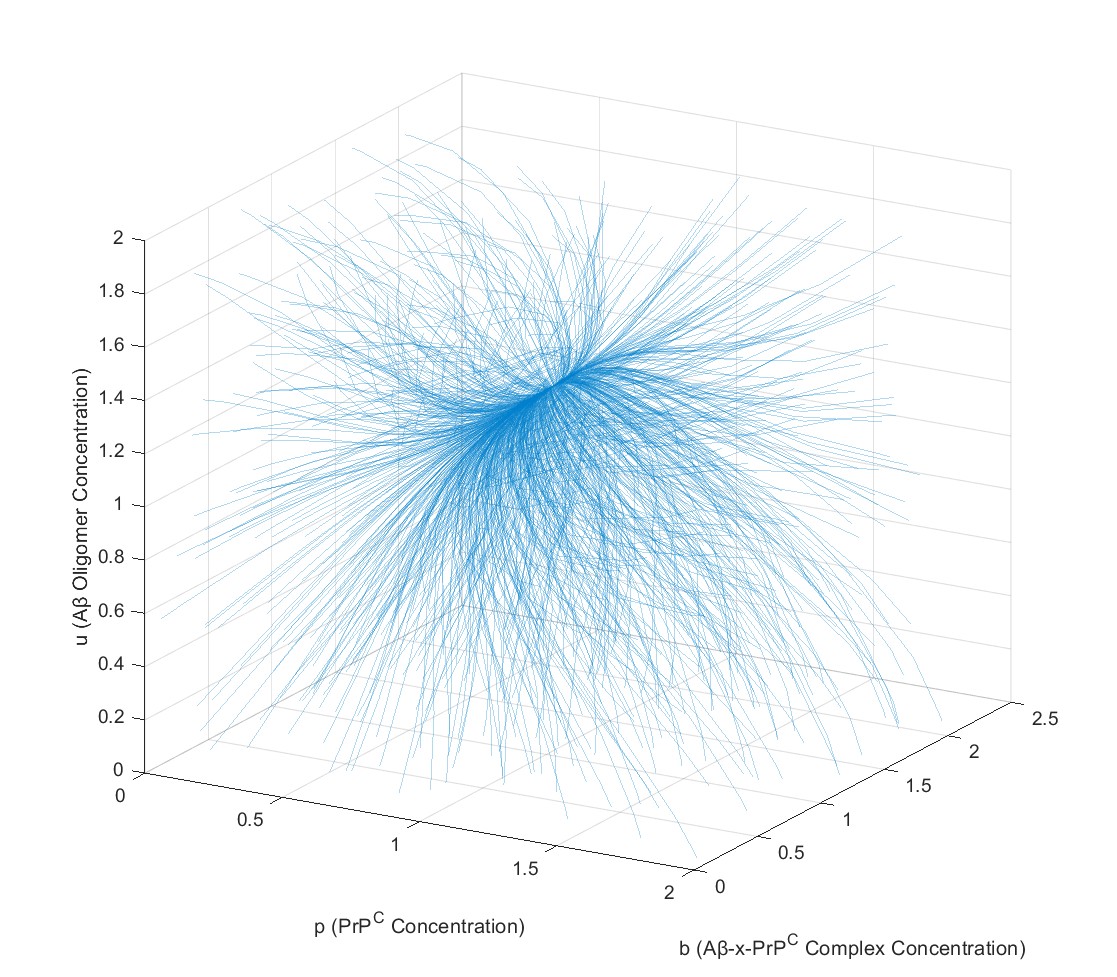}
		\caption{}
		\label{sub2}
	\end{subfigure}
	\caption{(a) 3D distribution of initial points, final states, and equilibrium in p-b-u space (1000 simulations). (b) 3D trajectories of $p-b-u$ (1000 initial conditions). } 
	\label{fig1}  
\end{figure}

In Fig. \ref{fig1}(a), the blue scatter points represent the distribution of initial conditions, while the red points and green star symbol indicate the final states and equilibrium, respectively. The curves in Fig. \ref{fig1}(b) represent the evolutionary trajectories of the variables. The convergence in Fig. \ref{fig1} confirms that regardless of the initial combination of $p$, $u$, and $b$, the system inevitably evolves toward the pathological steady state, which aligns with the theoretical conclusion that the equilibrium is globally asymptotically stable when \(\alpha=0\). Therefore, results in Fig. \ref{fig1} clearly demonstrate that despite significant differences in initial conditions, all trajectories converg to a narrow region around the unique equilibrium \(\tilde{E}\) without condition (\textbf{H}) in \citep{helal2014}.

Notably, this convergence implies that even if new plaque nucleation is blocked, existing \(A\beta\) oligomers will continuously bind to \(PrP^C\) to form neurotoxic complex. This provides a theoretical basis for understanding the dynamic regulation of $A\beta$ and $PrP^{C}$ in Alzheimer's disease, and also offers potential ideas for targeted intervention.

\subsection{Evaluation of Global Therapeutic Strategies}\label{sec4.2}	
In this subsection, we will turn to simulate the effects of different therapeutic strategies on disease progression.  In recent years, Alzheimer's disease drug treatment has shifted from long-term symptomatic improvement to $A\beta$-targeted immunotherapy, gradually entering a new phase of disease-modifying treatment (DMT) \citep{breijyeh2020,selkoe2016}. Among them, anti-$A\beta$ monoclonal antibodies primarily eliminate $A\beta$ deposits in the brain through immune clearance, including aducanumab, lecanemab, and donanemab \citep{hussain2024,sims2023,vandyck2023}. In addition, with the discovery of the mechanism by which $PrP^{C}$ binds to $A\beta$, targeting the cellular prion protein and the $A\beta$-x-$PrP^{C}$ complex have also entered the experimental stage \citep{moussa2021}. In the following text, we try to explore the effects of different drugs on Alzheimer's disease by changing the values of corresponding parameters. For convenience, in the following simulations, we take the initial values as ($A_0$, $u_0$, $p_0$, $b_0$) = (0, 0.2, 0.3, 0.2), and more generally, we choose $\alpha = 0.3$.

\begin{figure}[h]  
	\centering  
	\includegraphics[width=1\linewidth]{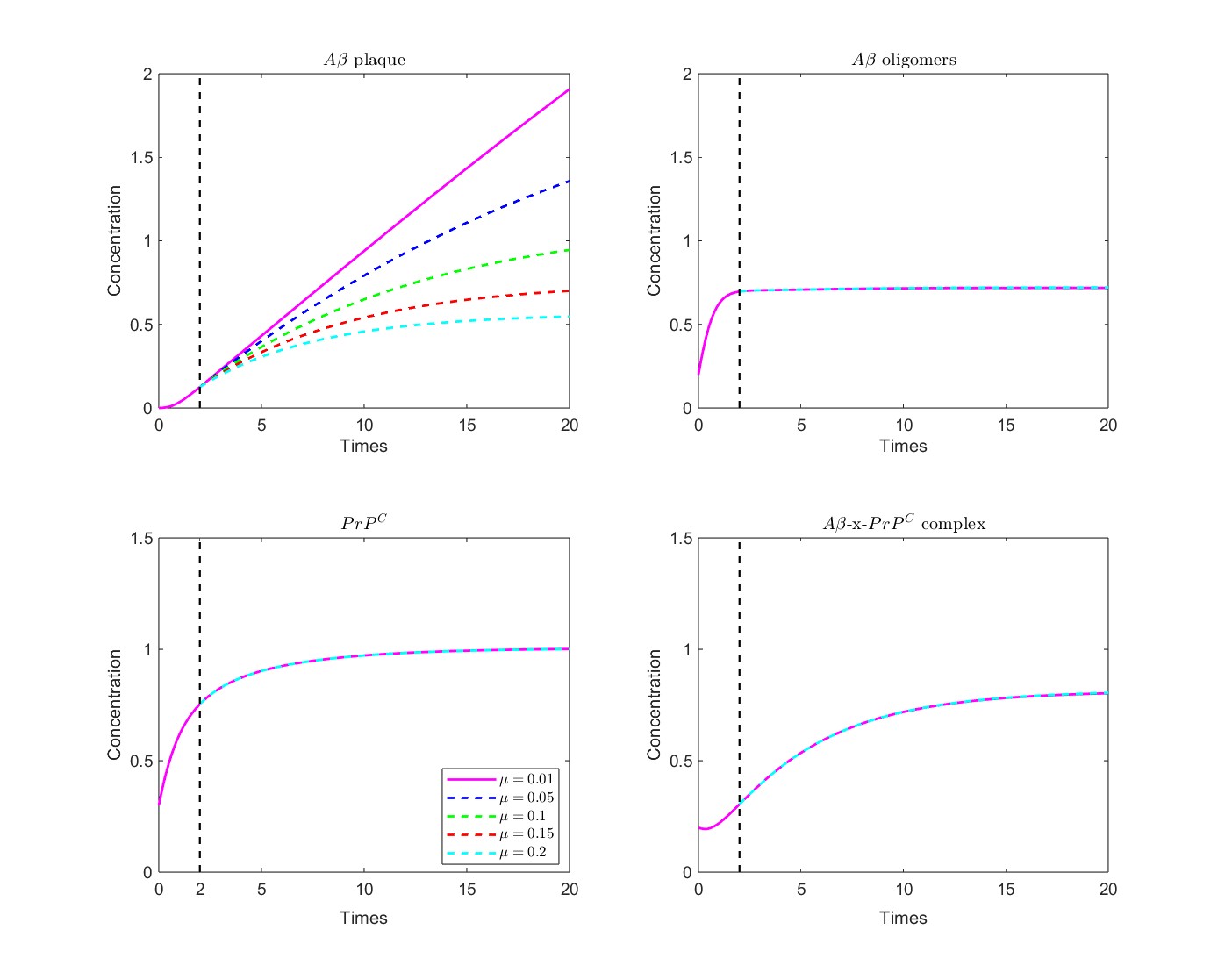}  
	\caption{Impacts of different degradation rate $\mu$ on the temporal dynamics of $A\beta$ plaques, $A\beta$ oligomers, $PrP^{C}$, and $A\beta$-x-$PrP^{C}$ complex}  
	\label{fig2} 
\end{figure}

To investigate the effects of anti-$A\beta$ monoclonal antibodies on the disease, Fig. \ref{fig2} and Fig. \ref{fig3} show simulations of the drug's clearance effect on $A\beta$ by altering the values of $\mu$ and $\gamma_{u}$ starting from $t=2$. As shown in Fig. \ref{fig2}, with $\mu$ increases, the clearance effect on $A\beta$ plaques becomes stronger with a slower cumulative effect, but it does not significantly alter the trends of other substances. Fig. \ref{fig3} demonstrates the effects of different $\gamma_{u}$ values on the time-dependent changes in the concentrations of key Alzheimer's disease molecules. As $\gamma_{u}$ increases, the accumulation rates of $A\beta$ plaques and $A\beta$-x-$PrP^{C}$ complex decrease, resulting in lower final concentrations. It also reduces the concentration of $A\beta$ oligomers, indicating that anti-$A\beta$ monoclonal antibodies have some efficiency on Alzheimer's disease intervention but have limited effects on $A\beta$-x-$PrP^{C}$ complex.

\begin{figure}[h]  
	\centering  
	\includegraphics[width=1\linewidth]{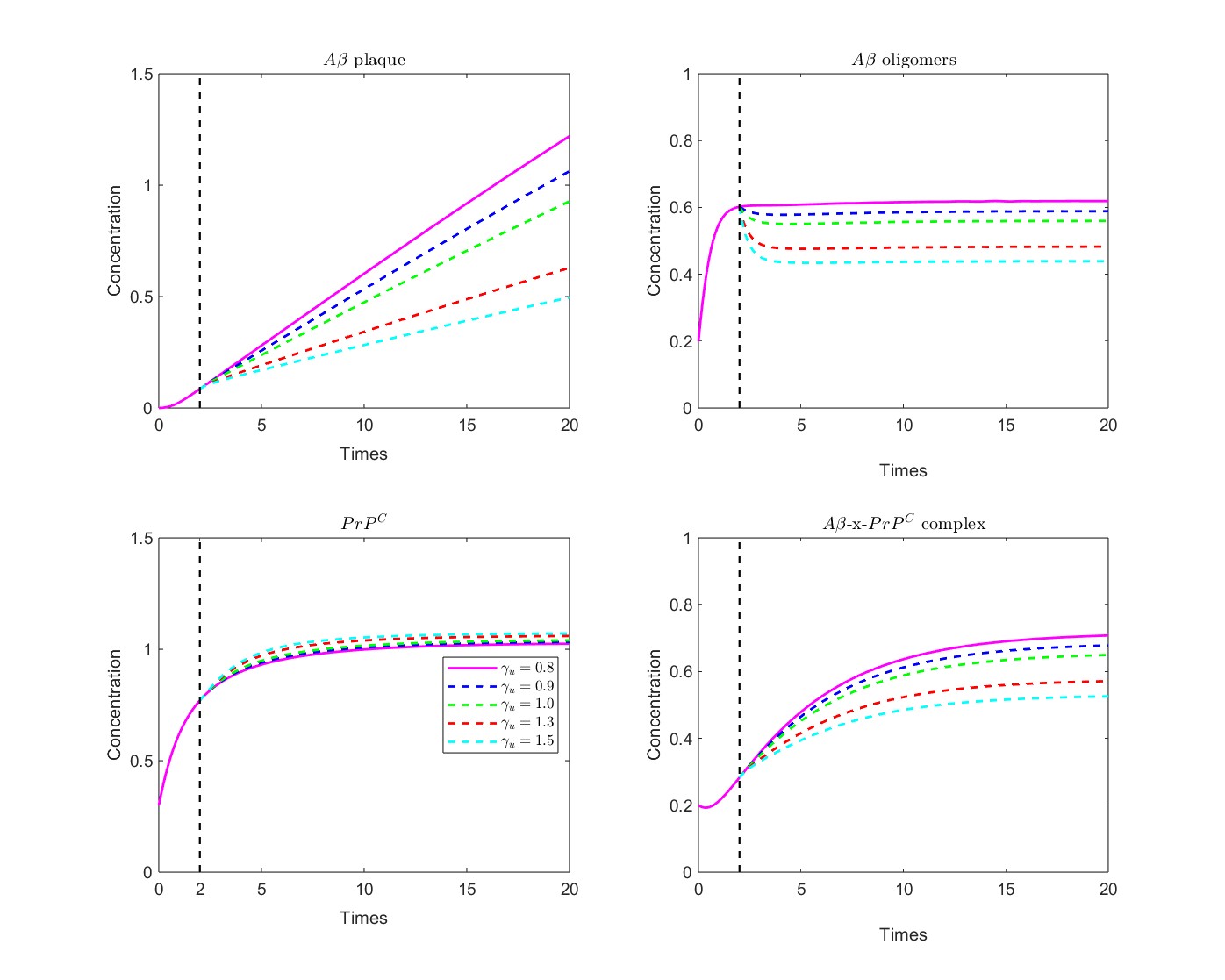}  
	\caption{Impacts of different degradation rate $\gamma_{u}$ values on $A\beta$ plaques, $A\beta$ oligomers, $PrP^{C}$, and $A\beta$-x-$PrP^{C}$ complex}  
	\label{fig3} 
\end{figure}

We simulate the effects of different \(\gamma_p\) and \(\tau\) on the dynamics of key molecules in Alzheimer's disease as shown in Fig. \ref{fig4} and Fig. \ref{fig5}. An increase in \(\gamma_p\) indicates targeting $PrP^{C}$, reducing $PrP^{C}$ concentration and thereby decreasing binding with \(A\beta\) oligomers. Simulations in Fig. \ref{fig4} show that this intervention has minimal effect on \(A\beta\) plaques or \(A\beta\) oligomers but more significantly decreases complex concentrations, as reduced \(PrP^C\) limits complex formation. A decrease in \(\tau\) indicates inhibition of complex formation, simulating a bispecific antibody (\(A\beta\)-$PrP^{C}$), where one end targets \(A\beta\) oligomers and the other end targets $PrP^{C}$. As shown in Fig. \ref{fig5}, the bispecific antibody has a more pronounced effect on reducing complex formation. These results reveal the targeting differences of different regulatory parameters, providing quantitative references for elucidating Alzheimer's disease pathological mechanisms and developing intervention strategies.

\begin{figure}[h]  
	\centering  
	\includegraphics[width=1\linewidth]{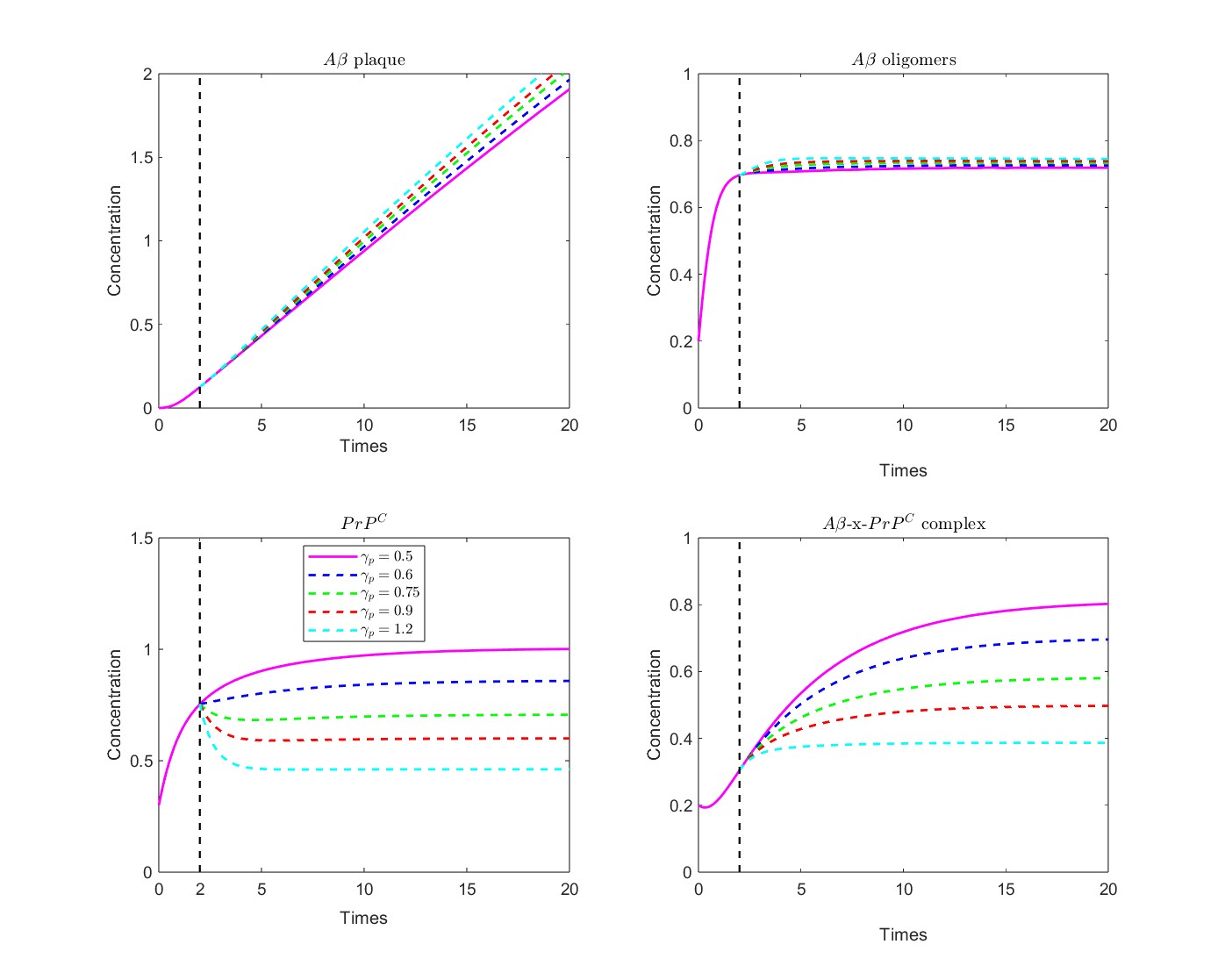}  
	\caption{Impacts of different degradation rate $\gamma_{p}$ on $A\beta$ plaques, $A\beta$ oligomers, $PrP^{C}$, and $A\beta$-x-$PrP^{C}$ complex}  
	\label{fig4} 
\end{figure}
\begin{figure}[h]  
	\centering  
	\includegraphics[width=1\linewidth]{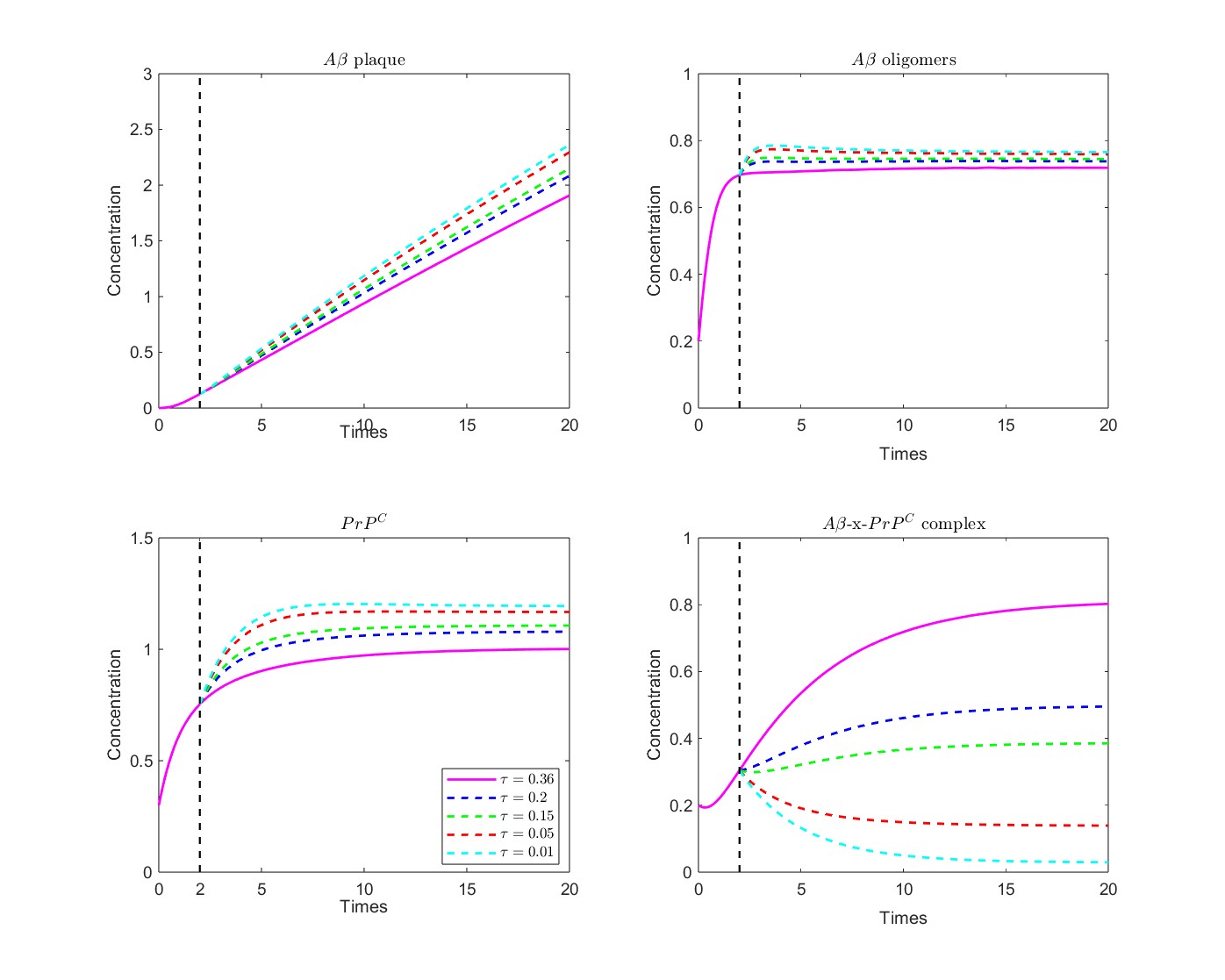}  
	\caption{The effect of different $\tau$ on the changes in A, u, p, and b}  
	\label{fig5} 
\end{figure}

Overall, these simulations reveal that while anti-\(A\beta\) therapies primarily reduce plaque burden, interventions targeting \(PrP^C\) or \(A\beta\)-\(PrP^C\) interactions are more effective at lowering neurotoxic complex. These findings support the development of multi-target strategies combining plaque clearance and complex inhibition are more promising for regulating pathological homeostasis.

\section{Discussion}\label{secdis}

Mathematical models have emerged as a powerful tool with which to understand the complex pathological mechanisms of Alzheimer's disease, particularly with regard to the dynamics of key molecules driving neurodegeneration. It is widely accepted that among the core pathological features, $\beta$-amyloid ($A\beta$) oligomers, cellular prion protein ($PrP^C$), and their toxic complex are critical mediators of synaptic dysfunction and neuronal loss. Previous studies have sought to develop models of the interactions between these components. Notably, Helal et al. \citep{helal2014} made a significant contribution to the research by developing models of molecular dynamics related to Alzheimer's disease. However, the analysis was dependent on a specific condition (\textbf{H}), which limits the generalizability of the conclusions. In the meantime, the majority of extant models concentrate exclusively on either the formation and clearance of $A\beta$ plaques or lack rigorous mathematical proof of the overall stability of the equilibrium state. The latter is a critical factor for understanding whether the pathological process tends towards a steady state.

Based on this, our study lies in clarifying the global asymptotic stability of the unique equilibrium in the (\ref{Aupb}) model under the condition of no new $A\beta$ plaque nucleation (i.e., $\alpha=0$). This result extends the prior work of Helal et al. \citep{helal2014} by eliminating the need for the specific condition (\textbf{H}), thereby enhancing the model’s applicability to more generalized Alzheimer's disease pathological scenarios. From a dynamical perspective, this explains that even with blocked new plaque formation, existing $A\beta$ oligomers continue binding \(PrP^C\), driving toxic complex accumulation and sustaining pathology in any situation. Theoretically, we demonstrate that monotonic dynamical systems theory and compound matrix methods provide a robust analytical framework for high-dimensional biological models. This framework not only resolves the global stability problem to the Alzheimer's disease model but also offers a reusable template for analyzing high-dimensional models of other neurodegenerative diseases.

In addition, consistent with clinical observations, simulations confirm that current drug-based interventions can alter the rate of disease progression,  but cannot fully halt pathological processes. For example, anti-$A\beta$ antibodies mainly reduce the buildup of $A\beta$ plaques and free $A\beta$ oligomers. While this directly affects a key pathological feature, it has limited impact on existing $A\beta$-x-$PrP^C$ complex, which are major drivers of neurotoxicity. In contrast, interventions targeting $PrP^C$ or the $A\beta$-x-$PrP^C$ complex itself more effectively reduce the overall levels of toxic species. These findings suggest that interfering with the interaction between $A\beta$ and $PrP^C$ may represent a more comprehensive therapeutic strategy than targeting $A\beta$ generation or plaque formation alone.

Although the assumption $\alpha = 0$ is useful for analysis, it oversimplifies the dynamic interactions between the formation and clearance of new plaques in progressive Alzheimer's disease. The stability of the system under conditions of active plaque nucleation has not yet been rigorously proven mathematically. Further research could be conducted to determine pathological progression and provide actionable insights for optimizing the timing of early intervention.

\bmhead{Acknowledgements}

	The authors are grateful to Prof. Yi Wang  in University of Science and Technology of China for his constructive suggestions and valuable guidance. 

\section*{Declarations}
\textbf{Conflicts of Interest} The authors declare that they have no conflicts of interest.

\vspace{7pt}
\noindent
\textbf{Ethics} No AI-assisted tools were used during the writing and revision of this manuscript.

\bibliography{sn-bibliography}

%{
%	\parskip=0.1pt        % 
%	\linespread{1}   % 
%	\setlength{\baselineskip}{9pt}
%	\bibliography{sn-bibliography}
%}

\end{document}